\documentclass[12pt]{article}
\usepackage[utf8]{inputenc}
\usepackage{graphicx}
\usepackage{amsmath}
\usepackage{color}
\usepackage{blkarray}
\usepackage{fullpage}
\usepackage{mathtools}
\usepackage{commath}
\usepackage{csquotes}
\usepackage{amsfonts}
\usepackage{amssymb}
\usepackage{amsthm}
\newtheorem{thm}{Theorem}[section]
\newtheorem{lem}{Lemma}[section]

\newtheorem{dfn}{Definition}[section]
\newtheorem{cor}{Corollary}[section]

\DeclareMathOperator{\diam}{diam}
\DeclareMathOperator{\tr}{trace}

\DeclareMathOperator{\Spec}{Spec}

\title{Minimizers for the energy of eccentricity matrices of trees}
\author{Iswar Mahato \thanks{Department of Mathematics, Indian Institute of Technology Kharagpur, Kharagpur 721302, India. Email: iswarmahato02@gmail.com}\  \and M. Rajesh Kannan\thanks{Department of Mathematics, Indian Institute of Technology Hyderabad, Hyderabad 502285, India. Email: rajeshkannan1.m@gmail.com, rajeshkannan@math.iith.ac.in }}
\date{\today}
\begin{document}
\maketitle

\begin{abstract}
	The eccentricity matrix of a connected graph $G$, denoted by $\mathcal{E}(G)$, is obtained from the distance matrix of $G$ by keeping the largest nonzero entries in each row and each column and leaving zeros in the remaining ones. The $\mathcal{E}$-eigenvalues of $G$ are the eigenvalues of $\mathcal{E}(G)$. The eccentricity energy (or the $\mathcal{E}$-energy) of $G$ is the sum of the absolute values of all $\mathcal{E}$-eigenvalues of $G$. In this article, we determine the unique tree with the minimum second largest $\mathcal{E}$-eigenvalue among all trees on $n$ vertices other than the star. Also, we characterize the trees with minimum $\mathcal{E}$-energy among all trees on $n$ vertices. 
\end{abstract}



\section{Introduction}\label{sec1}
Throughout this paper, we consider finite, simple and connected graphs. Let $G$ be a graph with vertex set $V(G)$ and edge set $E(G)$. The number of vertices in $G$ is the order of $G$. The \textit{adjacency matrix} of a graph $G$ on $n$ vertices, denoted by $A(G)$, is the $n \times n$ matrix whose rows and columns are indexed by the vertex set of $G$ and the entries are defined as
$$A(G)_{uv}=
\begin{cases}
	\text{$1$} & \quad\text{if $u$ and $v$ are adjacent,}\\
	\text{0} & \quad\text{otherwise.}
\end{cases}$$
The \textit{distance} $d_G(u,v)$ between the vertices $u,v\in V(G)$ is the length of a minimum path between them in $G$, and define $d(u,u) =0$  for all $u \in V(G)$. The \textit{distance matrix} $D(G)$ of $G$, is the $n \times n$ matrix with its rows and columns are indexed by the vertices of $G$ and the $(u,v)$-th entry is equal to $d_G(u,v)$. The \textit{eccentricity} $e_G(u)$ of a vertex $u\in V(G)$ is given by $e_G(u)=\max\{d_G(u,v):v\in V(G)\}$. The maximum eccentricity of all vertices of $G$ is the \textit{diameter} of $G$, which is denoted by $\diam(G)$. A \textit{diametrical path} in $G$ is a path whose length is equal to the diameter of $G$.

The \textit{ eccentricity matrix} $\mathcal{E}(G)$ of a graph $G$ of order $n$ is the $n\times n$ matrix indexed by the vertices of $G$ and the entries are defined as
$$\mathcal{E}(G)_{uv}=
\begin{cases}
	\text{$d_G(u,v)$} & \quad\text{if $d_G(u,v)=\min\{e_G(u),e_G(v)\}$,}\\
	\text{0} & \quad\text{otherwise.}
\end{cases}$$
The eccentricity matrix $\mathcal{E}(G)$ of a graph $G$ is a real symmetric matrix, and hence all of its eigenvalues are real. The eigenvalues of $\mathcal{E}(G)$ are the $\mathcal{E}$-eigenvalues of $G$, in which the largest one is the $\mathcal{E}$-spectral radius of $G$. 
Let $\xi_1>\xi_2>\hdots >\xi_k$ be all the distinct $\mathcal{E}$-eigenvalues of $G$. Then, the $\mathcal{E}$-spectrum of $G$ is defined as
\[ \Spec_{\varepsilon}(G)=
\left\{ {\begin{array}{cccc}
		\xi_1 & \xi_2  &\hdots & \xi_k\\
		m_1& m_2& \hdots &m_k\\
\end{array} } \right\},
\]
where $m_i$ is the  multiplicity of $\xi_i$ for $i=1,2,\hdots,k$. 

Although the eccentricity matrix is obtained from the distance matrix, some of the properties of the eccentricity matrix are substantially different from those of the distance matrix. For example, the distance matrix of a connected graph is always irreducible, whereas the eccentricity matrix of a connected graph need not be irreducible \cite{ecc-main}. The eccentricity matrices of trees need not be invertible but the distance matrices of trees are always invertible \cite{mahato2021spectral}. As the eccentricity matrices have different nature compared to the distance matrices, so it is interesting to study their spectral properties based on the combinatorial structure of the graphs. In \cite{ecc-main}, Wang et al. proved that the eccentricity matrices of trees are irreducible and investigated the relations between the $\mathcal{E}$-eigenvalues and the $A$-eigenvalues of graphs. Wang et al. \cite{wang2020spectral} obtained some bounds for the $\mathcal{E}$-spectral radius of graphs and determined the corresponding extremal graphs; Mahato et al. \cite{mahato2020spectra} investigated the spectra and the inertia of eccentricity matrices of various classes of graphs; Wei et al. \cite{wei2020solutions} studied the extremal problems for $\mathcal{E}$-spectral radius of trees and determined the trees on $n$ vertices with minimum $\mathcal{E}$-spectral radius; Lei et al. \cite{lei2021eigenvalues} characterized the graphs whose second least $\mathcal{E}$-eigenvalue is greater than $-\sqrt{15-\sqrt{193}}$ and proved that all these graphs are determined by their $\mathcal{E}$-spectrum. Recently, He and Lu \cite{he2022largest} Characterized the trees with maximum $\mathcal{E}$-spectral radius among all trees on $n$ vertices with fixed odd diameter. Wei et al. \cite{wei2022characterizing} gave an ordering of the trees with a given diameter regarding their $\mathcal{E}$-spectral radii and characterized the trees with second minimum $\mathcal{E}$-spectral radius. Very recently, Mahato and Kannan \cite{mahato2022inertia} determined the inertia of eccentricity matrices of trees and characterized the trees for which the $\mathcal{E}$-eigenvalues are symmetric about the origin. For more advances on the eccentricity matrices of graphs, we refer to \cite{andjelic2022extended,lei2022spectral,mahato2022eccentricity,patel2021energy,qiu2022eccentricity,wang2022spectraldetermination,wei2022eccentricity}. 

One of the main applications of eccentricity matrix is in chemical graph theory. Randi\'c, Orel, and Balaban \cite{ran2} indicated that the eccentricity matrix appears very sensitive to the branching pattern of molecular graphs studied in theoretical chemistry. In \cite{ecc-main}, the authors gave some applications of this novel matrix in terms of molecular descriptors. Recently, Wang et al. \cite{wang2020boiling} described some applications of eccentricity matrix to the boiling point of hydrocarbons. For more details about the applications of eccentricity matrices of graphs, we refer to \cite{ran1,ran2,wang2020boiling,ecc-main}. 

The second largest eigenvalues of the adjacency, Laplacian, and distance matrix of a graph are important graph invariants in spectral graph theory and are well studied in the literature. In this article, we consider the second largest eigenvalue of eccentricity matrices of trees and characterize the trees with the minimum second largest $\mathcal{E}$-eigenvalue.

In \cite{wang2019graph}, Wang et al. introduced the  \emph{eccentricity energy (or the $\mathcal{E}$-energy)} of a graph $G$, which is defined  as 
$$E_{\mathcal{E}}(G)=\sum_{i=1}^n |\xi_i(G)|,$$ 
where $\xi_1(G)\geq \xi_2(G)\geq \hdots \geq \xi_n(G)$ are the  $\mathcal{E}$--eigenvalues of $G$. Since $\sum_{i=1}^n \xi_i(G)=\tr(\mathcal{E}(G))=0$, the $\mathcal{E}$-energy of $G$ is $E_{\mathcal{E}}(G)=2\sum_{\xi_i>0}\xi_i(G)$. 

In \cite{wang2019graph}, Wang et al. studied the $\mathcal{E}$-energy of graphs and determined the $\mathcal{E}$-energies of paths, cycles, and double stars. Also, they obtained some bounds for the $\mathcal{E}$-energy of graphs and determined the corresponding extremal graphs. Mahato et al. \cite{mahato2021spectral} constructed a pair of graphs with the same $\mathcal{E}$-energy but different $\mathcal{E}$-spectrum for every $n\geq 5$. In \cite{patel2021energy}, the authors studied the $\mathcal{E}$-energy of the coalescence of complete graphs and the coalescence of two cycles. Lei et al. \cite{lei2021eigenvalues} obtained an upper bound for the $\mathcal{E}$-energy of graphs and characterized the extremal graphs. Recently, Mahato and Kannan \cite{mahato2022eccentricity} studied the $\mathcal{E}$-energy change of complete multipartite graphs due to an edge deletion and proved that the $\mathcal{E}$-energy of a complete multipartite graph always increases due to an edge deletion. Motivated by these works, in this article, we study the $\mathcal{E}$-energy of trees and characterize the trees with minimum $\mathcal{E}$-energy among all trees on $n$ vertices.

This article is organized as follows: In section $2$, we introduce some notations and collect some preliminary results. In section $3$, we identify the tree with the minimum second largest $\mathcal{E}$-eigenvalue among all trees on $n$ vertices other than the star. In section $4$, we characterize the trees with minimum $\mathcal{E}$-energy among all trees on $n$ vertices.

\section{Preliminaries}
For an $n\times n$ matrix $A$ with real entries, let $A^T$, $\det (A)$, $\tr(A)$ denote the transpose, the determinant and the trace of $A$, respectively. The \textit{spectrum} of an $n\times n$ matrix $A$ is the multi-set of all eigenvalues of $A$. The \textit{inertia} of a real symmetric matrix $A$ is the triple $\big(n_{+}(A),n_{-}(A),n_{0}(A)\big)$, where $n_{+}(A),$ $n_{-}(A)$ and $n_{0}(A)$ denote the number of positive, negative and zero eigenvalues of $A$, respectively. For a real number $x$, $\lfloor x \rfloor$ denotes the greatest integer less than or equal to $x$, and $\lceil x \rceil$ denotes the least integer greater than or equal to $x$. In the following theorem, we state the well-known Interlacing Theorem.

\begin{lem}[{\cite{hor-john-mat}, Interlacing Theorem}]
	Let $M$ be a real symmetric matrix of order $m$, and let $N$ be a principal submatrix of order $n<m$. If $\lambda_1\geq \lambda_2\geq \hdots \geq \lambda_m$ are the eigenvalues of $M$ and $\mu_1\geq \mu_2\geq \hdots \geq \mu_n$ are the eigenvalues of $N$, then $\lambda_{m-n+i} \leq \mu_i \leq \lambda_i$ for $1\leq i \leq n$.     
\end{lem}
Now, we recall the definition of quotient matrix and equitable partition.
\begin{dfn}[{\cite{brou-haem-book}, Equitable partition}] Let $A$ be a real symmetric matrix whose rows and columns are indexed by the set  $X=\{1,2,\hdots,n\}$. Let $\pi=\{X_1,X_2,\hdots,X_m\}$ be a partition of $X$. The \emph{characteristic matrix} $C$ is the $n\times m$ matrix whose $j$-th column is the characteristic vector of $X_j$ $(j=1,2,\hdots,m)$. Let $A$ be partitioned according to $\pi$ as follows \[A=\left[ {\begin{array}{cccc}
			A_{11} & A_{12} &\hdots & A_{1m}\\
			A_{21} & A_{22} &\hdots & A_{2m}\\
			\vdots &\hdots & \ddots & \vdots\\
			A_{m1} & A_{m2}& \hdots &A_{mm}\\
	\end{array} } \right],\]
	where $A_{ij}$ denotes the submatrix (block) of $A$ formed by rows in $X_i$ and the columns in $X_j$. If $q_{ij}$ denotes the average row sum of $A_{ij}$, then the matrix $Q=(q_{i,j})$ is  the \emph{quotient matrix} of $A$. If the row sum of each block $A_{ij}$ is constant, then the partition $\pi$ is the \emph{equitable partition}.
\end{dfn}

The following theorem is known for the spectrum of a quotient matrix corresponding to an equitable partition.

\begin{thm}[\cite{brou-haem-book}]\label{quo-spec}
	Let $Q$ be the quotient matrix of any square matrix $A$ corresponding to an equitable partition. Then the spectrum of $A$ contains the spectrum of $Q$.
\end{thm}

Let $K_n$, $P_n$ and $K_{1,n-1}$ denote the complete graph, the path and the star on $n$ vertices, respectively. Let $\mathcal{T}_{n,d}$ denote the set of all trees on $n$ vertices with diameter $d$. It is easy to see that the tree with diameter $1$ is $K_2$, and the tree with diameter $2$ is a star with at least $3$ vertices. For odd $d\geq 3$, let $T_{n,d}^{a,b}$ be the tree obtained from $P_{d+1}=v_0v_1v_2\hdots v_d$ by attaching $a$ pendant vertices to $v_{\frac{d-1}{2}}$ and $b$ pendent vertices to $v_{\frac{d+1}{2}}$, where $a+b=n-d-1$ and $b\geq a\geq 0$. If $d\geq 4$ is even, let $T_{n,d}^{a,b,c}$ be the tree obtained from $P_{d+1}=v_0v_1v_2\hdots v_d$ by attaching $a,b,c$ pendant vertices to $v_{\frac{d}{2}-1},v_{\frac{d}{2}}, v_{\frac{d}{2}+1}$, respectively, where $a+b+c=n-d-1$ and $c\geq a\geq 0, b\geq 0$. Now, we collect some useful results known for eccentricity matrices of graphs. 

\begin{lem}[{\cite[Theorem 2.1]{mahato2020spectra}}]\label{star-spec}
	Let $K_{1,n-1}$ be the star on $n$ vertices. Then the $\mathcal{E}$-spectrum of $K_{1,n-1}$ is given by
	\[\Spec_{\varepsilon}(K_{1,n-1})=
	\left\{ {\begin{array}{ccc}
			n-2+\sqrt{n^2-3n+3}  & n-2-\sqrt{n^2-3n+3} & -2  \\
			1 & 1 & n-2 \\
	\end{array} } \right\}.\]
\end{lem}

\begin{lem}[{\cite[Lemma 2.7]{wang2019graph}}]\label{dstar-spec}
	For $a+b=n-4$, $b\geq a \geq 0$, the $\mathcal{E}$-spectrum of $T_{n,3}^{a,b}$ is given by 
	\[\Spec_{\varepsilon}(T_{n,3}^{a,b})=
	\left\{ {\begin{array}{ccccc}
			\sqrt{\frac{\alpha+\sqrt{\beta}}{2}} & \sqrt{\frac{\alpha-\sqrt{\beta}}{2}} & 0 & -\sqrt{\frac{\alpha-\sqrt{\beta}}{2}} & -\sqrt{\frac{\alpha+\sqrt{\beta}}{2}} \\
			1 & 1 & n-4  & 1 & 1\\
	\end{array} } \right\},\]
	where $\alpha=9ab+13a+13b+17$ and $\beta=(9ab+13a+13b+17)^2-64(a+1)(b+1)$.		
\end{lem}

\begin{thm}[{\cite[Theorem 3.1]{mahato2022inertia}}]\label{Inertia-odd}
	Let $T$ be a tree on $n\geq 4$ vertices with $\diam(T)=2d+1$, $d\in \mathbb{N}$. Then $\mathcal{E}(T)$ has exactly two positive and two negative eigenvalues, that is, the inertia of $\mathcal{E}(T)$ is $(2,2,n-4)$.
\end{thm}

\begin{thm}[{\cite[Theorem 3.2]{mahato2022inertia}}]\label{inertia-even}
	Let $T$ be a tree on $n$ vertices with $\diam(T)=2d$, $d\geq 2, d\in \mathbb{N}$, and let $u_0$ be the centre of $T$. Let $u_1,u_2,\hdots, u_l$ $(l\geq 2)$ be the neighbours of $u_0$ such that $e_{T_i}(u_i)=d-1$, where $T_i$ are the components of $T-\cup_{i=1}^l u_0u_i$ containing $u_i$, for $i=1,2,\hdots, l$. Then the inertia of $\mathcal{E}(T)$ is $(l,l,n-2l)$.
\end{thm}

\begin{thm}[{\cite[Theorem 2.11]{wei2020solutions}}]\label{diam234-min}
	For $n\geq 4$, the tree $T_{n,3}^{0,n-4}$ is the unique tree with minimum $\mathcal{E}$-spectral radius among all trees in $\bigcup_{d=2}^4\mathcal{T}_{n,d}$. 
\end{thm}

\begin{lem}[{\cite[Lemma 2.6]{wei2020solutions}}]\label{odd-diam-min}
	Among $\mathcal{T}_{n,d}$ with odd $d\geq 5$, the minimum $\mathcal{E}$-spectral radius is achieved by some $T_{n,d}^{a,b}$, where $a+b=n-d-1$ and $b\geq a\geq 0$.
\end{lem}

\begin{thm}[{\cite[Theorem 2.15]{wei2020solutions}}]\label{min-odd}
	The tree $T_{n,d}^{\lfloor \frac{n-d-1}{2}\rfloor,\lceil \frac{n-d-1}{2}\rceil}$ is the unique tree with minimum $\mathcal{E}$-spectral radius among all trees in $\mathcal{T}_{n,d}$ with odd $d\geq 5$.
\end{thm}

\begin{lem}[{\cite[Lemma 2.8]{wei2020solutions}}]\label{even-diam-min}
	Among all trees in $\mathcal{T}_{n,d}$ with even $d\geq 6$, the minimum $\mathcal{E}$-spectral radius is achieved by some $T_{n,d}^{a,b,c}$, where $a+b+c=n-d-1$ and $c\geq a\geq 0, b\geq 0$.
\end{lem}

\begin{thm}[{\cite[Theorem 2.16]{wei2020solutions}}]\label{min-even}
	The tree $T_{n,d}^{\lfloor \frac{n-d-1}{2}\rfloor,0,\lceil \frac{n-d-1}{2}\rceil}$ is the unique tree with minimum $\mathcal{E}$-spectral radius among all trees in $\mathcal{T}_{n,d}$ with even $d\geq 6$.
\end{thm}

\begin{lem}[{\cite[Lemma 2.9]{wei2020solutions}}]\label{diam-red-odd}
	For odd $d\geq 7$, one has $\xi_1(T_{n,d-2}^{a+1,b+1})<\xi_1(T_{n,d}^{a,b})$.
\end{lem}

\begin{lem}[{\cite[ Lemma 2.10]{wei2020solutions}}]\label{diam-red-even}
	For even $d\geq 6$, one has $\xi_1(T_{n,d-1}^{a,b+c+1})<\xi_1(T_{n,d}^{a,b,c})$.
\end{lem}

\begin{thm}[{\cite{wei2020solutions}}]\label{spec odd}
	For odd $d\geq 5$, the $\mathcal{E}$-spectral radius of $T_{n,d}^{a,b}$ is the largest root of 
	$$\left| {\begin{array}{cc}
			{\rho}^2-\Gamma(d)-\big(\frac{d+3}{2}\big)^2b & -d\rho \\
			-d\rho & {\rho}^2-\Gamma(d)-\big(\frac{d+3}{2}\big)^2a\\
	\end{array} } \right|=0, $$
	where $\Gamma(d)=\frac{d(d-1)(7d-5)}{24}$.
\end{thm}

\begin{thm}[{\cite{wei2020solutions}}]\label{spec even}
	For even $d\geq 6$, the $\mathcal{E}$-spectral radius of $T_{n,d}^{a,b,c}$ is the largest root of 
	$$\left| {\begin{array}{cc}
			{\rho}^2-\Theta(d)-\big(\frac{d+2}{2}\big)^2b-\big(\frac{d+4}{2}\big)^2c & -d\rho-\big(\frac{d}{2}\big)^2-\big(\frac{d+2}{2}\big)^2b \\
			-d\rho-\big(\frac{d}{2}\big)^2-\big(\frac{d+2}{2}\big)^2b & {\rho}^2-\Theta(d)-\big(\frac{d+2}{2}\big)^2b-\big(\frac{d+4}{2}\big)^2a\\
	\end{array} } \right|=0, $$
	where $\Theta(d)=\frac{d(d-1)(7d-2)}{24}$.
\end{thm}

\begin{thm}[{\cite[Theorem 2.13]{wei2020solutions}}]\label{min-gen}
	Let $T$ be a tree on $n$ vertices.
	\begin{enumerate}
		\item If $4 \leq n\leq 15$, then 
		$$\xi_1(T)\geq \sqrt{\frac{13n-35+\sqrt{(13n-35)^2-64(n-3)}}{2}} $$ with equality if and only if $T\cong T_{n,3}^{0,n-4}$.
		\item If $n\geq 16$, then $$\xi_1(T)\geq
		\begin{cases}
			\text{$\sqrt{\frac{16n-21+\sqrt{800n-1419}}{2}}$,} & \quad\text{if $n$ is odd;}\\
			\text{$\sqrt{\frac{16n-21+5\sqrt{32n-67}}{2}}$,} & \quad \text{if $n$ is even.}
		\end{cases} $$
		Each of the equalities holds if and only if $T\cong T_{n,5}^{\lfloor \frac{n-6}{2}\rfloor,\lceil \frac{n-6}{2}\rceil}$.
	\end{enumerate}
\end{thm}

\begin{thm}[{\cite[Theorem 3.4]{wei2022characterizing}}]\label{ordering-odd}
	If $n\geq d+2$ with odd $d\geq 5$, then all the trees among $\{T_{n,d}^{a,b}:b\geq a\geq 0 \}$ can be ordered with respect to $\xi_1$ as 
	$$\xi_1(T_{n,d}^{\lfloor \frac{n-d-1}{2}\rfloor,\lceil \frac{n-d-1}{2}\rceil}) <\xi_1(T_{n,d}^{\lfloor \frac{n-d-1}{2}\rfloor-1,\lceil \frac{n-d-1}{2}\rceil+1}) <\xi_1(T_{n,d}^{\lfloor \frac{n-d-1}{2}\rfloor-2,\lceil \frac{n-d-1}{2}\rceil+2}) < \hdots < \xi_1(T_{n,d}^{0,n-d-1}).$$
\end{thm}

\begin{thm}[{\cite[Theorem 3.6]{wei2022characterizing}}]\label{ordering-even}
	Assume that $n\geq d+2$ with even $d\geq 6$.
	\begin{enumerate}
		\item If $n$ is odd, then all the trees among $\{T_{n,d}^{a,b,c}:c\geq a\geq 0,b\geq 0 \}$ can be ordered with respect to $\xi_1$ as 
		\begin{multline*}
			\xi_1(T_{n,d}^{\frac{n-d-1}{2},0,\frac{n-d-1}{2}})< \xi_1(T_{n,d}^{\frac{n-d-3}{2},1,\frac{n-d-1}{2}})< \xi_1(T_{n,d}^{\frac{n-d-3}{2},2,\frac{n-d-3}{2}})< \xi_1(T_{n,d}^{\frac{n-d-5}{2},3,\frac{n-d-3}{2}})\\ < \xi_1(T_{n,d}^{\frac{n-d-5}{2},4,\frac{n-d-5}{2}})<  \hdots < \xi_1(T_{n,d}^{0,n-d-2,1})< \xi_1(T_{n,d}^{0,n-d-1,0});
		\end{multline*}
		\item If $n$ is even, then all the trees among $\{T_{n,d}^{a,b,c}:c\geq a\geq 0,b\geq 0 \}$ can be ordered with respect to $\xi_1$ as 
		\begin{multline*}
			\xi_1(T_{n,d}^{\frac{n-d-2}{2},0,\frac{n-d}{2}})< \xi_1(T_{n,d}^{\frac{n-d-2}{2},1,\frac{n-d-2}{2}})< \xi_1(T_{n,d}^{\frac{n-d-4}{2},2,\frac{n-d-2}{2}})< \xi_1(T_{n,d}^{\frac{n-d-4}{2},3,\frac{n-d-4}{2}})\\ < \xi_1(T_{n,d}^{\frac{n-d-6}{2},4,\frac{n-d-4}{2}})<  \hdots < \xi_1(T_{n,d}^{0,n-d-2,1})< \xi_1(T_{n,d}^{0,n-d-1,0}).
		\end{multline*}
	\end{enumerate} 
\end{thm}

\section{Second largest $\mathcal{E}$-eigenvalue of trees}
In this section, we characterize the trees with minimum second largest $\mathcal{E}$-eigenvalue among all trees on $n$ vertices other than the star. First, let us determine the unique tree with minimum second largest $\mathcal{E}$-eigenvalue among the trees with diameter $3$.

\begin{lem}\label{second-diam3}
Let $T$ be a tree on $n\geq 4$ vertices with diameter $3$. Then,
$$\xi_2(T)\geq \sqrt{\frac{13n-35-\sqrt{169n^2-974n+1417}}{2}}$$
with equality if and only if $T\cong T_{n,3}^{0,n-4}$.
\end{lem}
\begin{proof}
Since $T$ is a tree of diameter $3$, therefore $T\cong T_{n,3}^{a,b}$. Therefore, by Lemma \ref{dstar-spec}, it follows that $$\xi_2(T)=\xi_2(T_{n,3}^{a,b})= \sqrt{\frac{(9ab+13a+13b+17)-\sqrt{(9ab+13a+13b+17)^2-64(a+1)(b+1)}}{2}}.$$

Since $a+b=n-4$ $(0\leq a \leq \lfloor \frac{n-4}{2}\rfloor)$, therefore 

$$\xi_2(T)=\sqrt{\frac{9a(n-4-a)+13n-35-\sqrt{(9a(n-4-a)+13n-35)^2-64(a(n-4-a)+n-3)}}{2}}.$$

If $a=0$, then $T\cong T_{n,3}^{0,n-4}$ and  $\xi_2(T_{n,3}^{0,n-4})=\sqrt{\frac{13n-35-\sqrt{169n^2-974n+1417}}{2}}$. Now, consider the function
 
 $$f(x)=9x(n-4-x)+13n-35-\sqrt{(9x(n-4-x)+13n-35)^2-64(x(n-4-x)+n-3)}.$$
 Therefore,
\begin{equation*}
f^{\prime}(x)= 9(n-4-2x)\bigg(1-\frac{9x(n-4-x)+(13n-35)-\frac{32}{9}}{\sqrt{(9x(n-4-x)+13n-35)^2-64(x(n-4-x)+n-3)}}\bigg). 
\end{equation*}

Since $n\geq 4$, therefore 
\begin{eqnarray*}
& &\sqrt{(9x(n-4-x)+13n-35)^2-64(x(n-4-x)+n-3)} \\
&=& \sqrt{\Big(9x(n-4-x)+13n-35-\frac{32}{9}\Big)^2+\frac{2304n-5632}{81}}\\
&>& 9x(n-4-x)+13n-35-\frac{32}{9}.
\end{eqnarray*}

Hence, $$\frac{9x(n-4-x)+(13n-35)-\frac{32}{9}}{\sqrt{(9x(n-4-x)+13n-35)^2-64(x(n-4-x)+n-3)}}<1.$$ 

Thus, $f^{\prime}(x)\geq 0$ for $0\leq x \leq \lfloor \frac{n-4}{2}\rfloor$, that is, $f(x)$ is an increasing function of $x$ for $0\leq x \leq \lfloor \frac{n-4}{2}\rfloor$. Hence,
\begin{eqnarray*}
\xi_2(T)=\sqrt{\frac{f(a)}{2}}\geq \sqrt{\frac{f(0)}{2}} &=&
\sqrt{\frac{13n-35-\sqrt{(13n-35)^2-64(n-3)}}{2}}\\
&=& \sqrt{\frac{13n-35-\sqrt{169n^2-974n+1417}}{2}}
\end{eqnarray*}
and the equality holds if and only if $a=0$, that is, $T\cong T_{n,3}^{0,n-4}$.
\end{proof}

From the proof of Lemma \ref{second-diam3}, we can order the trees of diameter $3$ according to their second largest $\mathcal{E}$-eigenvalues.
\begin{cor}\label{diam3-sec-ordering}
Let $T_{n,3}^{a,b}$ be in $\mathcal{T}_{n,3}$ with $b\geq a\geq 0$. Then all the trees in $T_{n,3}^{a,b}$ can be ordered with respect to their second largest $\mathcal{E}$-eigenvalues as follows:  $$\xi_2(T_{n,3}^{0,n-4})<\xi_2(T_{n,3}^{1,n-5})<\hdots<\xi_2(T_{n,3}^{\lfloor \frac{n-4}{2}\rfloor,\lceil \frac{n-4}{2}\rceil}).$$
\end{cor}

Next, we obtain an upper bound for the second largest $\mathcal{E}$-eigenvalue of $T_{n,3}^{a,b}$.

\begin{lem}\label{bound-second}
For $n\geq 4$, $\xi_2(T_{n,3}^{0,n-4})< \sqrt{2}$.
\end{lem}
\begin{proof}
Note that,
$\sqrt{(13n-37)^2-12(n-4)}>\sqrt{(13n-37)^2-(13n-37)}=\sqrt{(13n-37)(13n-38)}$ $> (13n-38)$. Thus,
\begin{eqnarray*}
\xi_2(T_{n,3}^{0,n-4})&=& \sqrt{\frac{13n-35-\sqrt{169n^2-974n+1417}}{2}}\\
&=& \sqrt{\frac{13n-35-\sqrt{(13n-37)^2-12(n-4)}}{2}}\\
&<& \sqrt{\frac{13n-34-(13n-38)}{2}}\\
&=& \sqrt{2}.
\end{eqnarray*}
\end{proof}

In the following theorem, we show that $T_{n,3}^{0,n-4}$ is the unique tree with the minimum second largest $\mathcal{E}$-eigenvalue among all trees on $n$ vertices other than the star.
\begin{thm}\label{second-larg}
Let $T$ be a tree on $n\geq 4$ vertices other than the star. Then, 
$$\xi_2(T)\geq \xi_2(T_{n,3}^{0,n-4}) $$
with equality if and only if $T\cong T_{n,3}^{0,n-4}$.
\end{thm}
\begin{proof}
If $T$ is a tree with $\diam(T)=3$, then the proof follows from Theorem \ref{second-diam3}. Let $T$ be a tree with $\diam(T)\geq 4$, and let $P_{d+1}=v_0v_1\hdots v_d$ be a diametrical path in $T$. Then, the principal submatrix of $\mathcal{E}(T)$ indexed by $\{v_0,v_1,v_{d-1},v_d\}$ is 
\[B = \left[ {\begin{array}{cccc}
0 & 0 & d-1 & d\\
0 & 0 & 0 & d-1\\
d-1 & 0 & 0 & 0 \\
d & d-1 & 0 & 0 \\
\end{array} } \right].\]
 By a direct calculation, the second largest eigenvalue of $B$ is given by $\lambda_2=\frac{\sqrt{d^2+4(d-1)^2}-d}{2}$. Therefore, by Cauchy Interlacing Theorem, we have $\xi_2(T)\geq \lambda_2=\frac{\sqrt{d^2+4(d-1)^2}-d}{2}$. Now, we claim that $\lambda_2>\sqrt{2}$. Suppose that  $\lambda_2=\frac{\sqrt{d^2+4(d-1)^2}-d}{2}<\sqrt{2}$. Then $d^3(d-4)+4d+1<0$, which is a contradiction as $d\geq 4$. Therefore, $\xi_2(T)\geq \lambda_2 > \sqrt{2}$. Now, the proof follows from Lemma \ref{bound-second}.
\end{proof}

\section{Minimum $\mathcal{E}$-energy of trees}
In this section, we consider the extremal problem for the $\mathcal{E}$-energy of trees. We show that the tree $T_{n,3}^{0,n-4}$ is the unique tree with minimum $\mathcal{E}$-energy among all trees on $n\geq 5$ vertices. First, we compute the $\mathcal{E}$-energy of $T_{n,3}^{0,n-4}$ and then compare it with the $\mathcal{E}$-energy of $T_{n,5}^{a,b},T_{n,6}^{a,b,a},T_{n,6}^{a,b,a+1}$ and $T_{n,7}^{a,b}$, respectively.

\begin{lem}\label{energy-bound}
For $n\geq 5$, the $\mathcal{E}$-energy of the tree $T_{n,3}^{0,n-4}$ is given by $$E_{\mathcal{E}}(T_{n,3}^{0,n-4})=2\sqrt{13n-35+8\sqrt{n-3}}.$$
\end{lem}
\begin{proof}
Since $T_{n,3}^{0,n-4}$ is a tree with diameter $3$, therefore by Theorem \ref{Inertia-odd}, it has exactly two positive $\mathcal{E}$-eigenvalues. Hence, $E_{\mathcal{E}}(T_{n,3}^{0,n-4}) = 2\big(\xi_1(T_{n,3}^{0,n-4})+\xi_2(T_{n,3}^{0,n-4})\big)$. Again, by Lemma \ref{dstar-spec}, we have 
\begin{align*}
 \xi_1(T_{n,3}^{0,n-4})=& \sqrt{\frac{13n-35+\sqrt{169n^2-974n+1417}}{2}} \quad \text{and}\\ \quad  \xi_2(T_{n,3}^{0,n-4})=& \sqrt{\frac{13n-35-\sqrt{169n^2-974n+1417}}{2}}. 
\end{align*}
Therefore, $\big(\xi_1(T_{n,3}^{0,n-4})\big)^2+\big(\xi_2(T_{n,3}^{0,n-4})\big)^2=13n-35$ and $\xi_1(T_{n,3}^{0,n-4})\xi_2(T_{n,3}^{0,n-4})=4\sqrt{n-3}$. Hence, 
\begin{align*}
\Big(\xi_1(T_{n,3}^{0,n-4})+\xi_2(T_{n,3}^{0,n-4})\Big)^2=& \big(\xi_1(T_{n,3}^{0,n-4})\big)^2+\big(\xi_2(T_{n,3}^{0,n-4})\big)^2+2\xi_1(T_{n,3}^{0,n-4})\xi_2(T_{n,3}^{0,n-4})\\
=& 13n-35+8\sqrt{n-3}.
\end{align*}
Thus, $E_{\mathcal{E}}(T_{n,3}^{0,n-4}) = 2\big(\xi_1(T_{n,3}^{0,n-4})+\xi_2(T_{n,3}^{0,n-4})\big)=2\sqrt{13n-35+8\sqrt{n-3}}$. 
\end{proof}

\begin{thm}\label{diam5-ener}
For $b\geq a \geq 1$ and $a+b=n-6$, $E_{\mathcal{E}}(T_{n,5}^{a,b})>E_{\mathcal{E}}(T_{n,3}^{0,n-4})$.
\end{thm}
\begin{proof}
First, let us calculate the $\mathcal{E}$-spectrum of $T_{n,5}^{a,b}$, where $a+b=n-6$ with $b\geq a \geq 1$. Let $P_6=v_0v_1v_2v_3v_4v_5$ be the unique diametrical path in $T_{n,5}^{a,b}$. Assume that $u_1,u_2,\hdots,u_a$ be the pendant vertices adjacent to $v_2$ and $w_1,w_2,\hdots,w_b$ be the pendant vertices adjacent to $v_3$. Let $V_{a+1}=\{v_1,u_1,u_2,\hdots,u_a\}$ and $V_{b+1}=\{v_4,w_1,w_2,\hdots,w_b\}$. It is clear that $\Pi_1:\{v_0\}\cup V_{a+1}\cup \{v_2\}\cup \{v_3\}\cup V_{b+1}\cup \{v_5\}$ is an equitable partition of $\mathcal{E}(T_{n,5}^{a,b})$ with the quotient matrix 
\[Q_{\Pi_1}=
\left [ {\begin{array}{cccccc}
0 & 0 & 0 & 3 & 4(b+1) & 5\\
0 & 0 & 0 & 0 & 0 & 4 \\
0 & 0 & 0 & 0 & 0 & 3 \\
3 & 0 & 0 & 0 & 0 & 0 \\
4 & 0 & 0 & 0 & 0 & 0 \\
5 & 4(a+1) & 3 & 0 & 0 & 0 \\
\end{array} } \right].\]
Now, the characteristic polynomial of $Q_{\Pi_1}$ is $$f(x)=x^2(x^4-16(a+b+75)x^2+256ab+400(a+b)+625).$$
Therefore, by Theorem \ref{quo-spec} and Theorem \ref{Inertia-odd}, the $\mathcal{E}$-spectrum of $T_{n,5}^{a,b}$ is given by 
\[\Spec_{\varepsilon}(T_{n,5}^{a,b})=
	\left\{ {\begin{array}{ccccc}
		\sqrt{\frac{\alpha+\sqrt{\beta}}{2}} & \sqrt{\frac{\alpha-\sqrt{\beta}}{2}} & 0 & -\sqrt{\frac{\alpha-\sqrt{\beta}}{2}} & -\sqrt{\frac{\alpha+\sqrt{\beta}}{2}} \\
		1 & 1 & n-4  & 1 & 1\\
		\end{array} } \right\},\]
where $\alpha=16a+16b+75$ and $\beta=256(a+b)^2+800(a+b)-1024ab+3125$. Thus,
\begin{align*}
\big(E_{\mathcal{E}}(T_{n,5}^{a,b})\big)^2=& 4\Bigg(\sqrt{\frac{\alpha+\sqrt{\beta}}{2}}+\sqrt{\frac{\alpha-\sqrt{\beta}}{2}}\Bigg)^2 \\
=& 4(\alpha+\sqrt{\alpha^2-\beta})\\
= & 4(16a+16b+75+\sqrt{1600(a+b)+1024ab+2500})\\
= & 4(16n-21+\sqrt{1600n-7100+1024ab})\qquad (\text{since,~} a+b=n-6).
\end{align*}
Again, by Lemma \ref{energy-bound}, we have $\big(E_{\mathcal{E}}(T_{n,3}^{0,n-4})\big)^2=4(13n-35+8\sqrt{n-3})$. Note that $16n-21+\sqrt{1600n-7100+1024ab}>13n-35+8\sqrt{n-3}$ for $n\geq 5$. Hence, $E_{\mathcal{E}}(T_{n,5}^{a,b})>E_{\mathcal{E}}(T_{n,3}^{0,n-4})$.
\end{proof}

\begin{lem}\label{diam6-ener1}
For $a\geq 1,b\geq 0$ and $2a+b=n-7$, $E_{\mathcal{E}}(T_{n,6}^{a,b,a})>E_{\mathcal{E}}(T_{n,3}^{0,n-4})$.
\end{lem}
\begin{proof}
First, let us calculate the $\mathcal{E}$-spectrum of $T_{n,6}^{a,b,a}$ for $2a+b=n-7$,$a\geq 1,b\geq 0$. Let $P_7=v_0v_1v_2v_3v_4v_5v_6$ be the unique diametrical path in $T_{n,6}^{a,b,a}$. For $i,k=1,2,\hdots,a$ and $j=1,2,\hdots,b$, let us assume that $u_i,y_j$ and $w_k$ be the pendant vertices adjacent to $v_2,v_3$ and $v_4$, respectively. Let $U_{a+1}=\{v_1,u_1,u_2,\hdots,u_a\}$, $V_b=\{y_1,y_2,\hdots,y_b\}$ and $V_{a+1}=\{v_4,w_1,w_2,\hdots,w_a\}$. It is clear that $\Pi_2:\{v_0\}\cup U_{a+1}\cup \{v_2\}\cup \{v_3\}\cup V_{b}\cup \{v_4\}\cup V_{a+1}\cup \{v_6\}$ is an equitable partition of $\mathcal{E}(T_{n,6}^{a,b,a})$ with the quotient matrix
\[Q_{\Pi_2}=
\left [ {\begin{array}{cccccccc}
0 & 0 & 0 & 3 & 4b & 4 & 5(a+1) & 6\\
0 & 0 & 0 & 0 & 0 & 0 & 0 & 5 \\
0 & 0 & 0 & 0 & 0 & 0 & 0 & 4 \\
3 & 0 & 0 & 0 & 0 & 0 & 0 & 3 \\
4 & 0 & 0 & 0 & 0 & 0 & 0 & 4 \\
4 & 0 & 0 & 0 & 0 & 0 & 0 & 0 \\
5 & 0 & 0 & 0 & 0 & 0 & 0 & 0 \\
6 & 5(a+1) & 4 & 3 & 4b & 0 & 0 & 0\\
\end{array} } \right].\]
Now, the characteristic polynomial of $Q_{\Pi_2}$ is $f(x)=x^4\big(x^4-(50a+32b+136)x^2-(192b+108)x+625a^2+2500a+800ab+1312b+2419\big)$. Therefore, from Theorem \ref{quo-spec} and Theorem \ref{inertia-even}, it follows that the $\mathcal{E}$-spectrum of $T_{n,6}^{a,b,a}$ is given by 
\[\Spec_{\varepsilon}(T_{n,6}^{a,b,a})=
	\left\{ {\begin{array}{ccc}
		3\pm \sqrt{25a+32b+68} & \pm 5\sqrt{a+2}-3 & 0\\
		1 & 1 & n-4 \\
		\end{array} } \right\}.\]
Hence,
\begin{align*}
 E_{\mathcal{E}}(T_{n,6}^{a,b,a})=&2(\sqrt{25a+32b+68}+5\sqrt{a+2})\\
 =&2(\sqrt{32n-39a-156}+5\sqrt{a+2}) \qquad \text{(since, $2a+b=n-7$)}.
\end{align*}
\textbf{Case 1}: If $a<\frac{19n-120}{39}$, then $\sqrt{32n-39a-156}>\sqrt{13n-36}$ and hence by Lemma \ref{bound-spec} and Lemma \ref{bound-second}, we have $E_{\mathcal{E}}(T_{n,6}^{a,b,a})=2(\sqrt{32n-39a-156}+5\sqrt{a+2})>2(\sqrt{13n-36}+\sqrt{2})>2(\xi_1(T_{n,3}^{0,n-4})+\xi_2(T_{n,3}^{0,n-4}))=E_{\mathcal{E}}(T_{n,3}^{0,n-4})$.\\
\textbf{Case 2}: Let $a>\frac{19n-120}{39}$. Since $b=n-2a-7\geq 0$, therefore
\begin{align*}
 E_{\mathcal{E}}(T_{n,6}^{a,b,a})=&2(\sqrt{32n-39a-156}+5\sqrt{a+2})\\
 >& 2\bigg(\sqrt{\frac{25n-39}{2}}+\sqrt{\frac{475n-1050}{39}}\bigg).   
\end{align*}
Hence,
\begin{align*}
&(E_{\mathcal{E}}(T_{n,6}^{a,b,a}))^2-(E_{\mathcal{E}}(T_{n,3}^{0,n-4}))^2\\
>& 4\bigg(\sqrt{\frac{25n-39}{2}}+\sqrt{\frac{475n-3000}{39}}\bigg)^2-4\bigg(\sqrt{13n-35+8\sqrt{n-3}}\bigg)^2\\
=&4\bigg(\frac{624(n-3-\sqrt{n-3})+287n+981}{78}+2\sqrt{\frac{11875n^2-93525n+117000}{78}}\bigg)\\
>& 0.
\end{align*}
\end{proof}

\begin{thm}\label{diam6-ener2}
For $a,b\geq 0$ and $2a+b=n-8$, $E_{\mathcal{E}}(T_{n,6}^{a,b,a+1})>E_{\mathcal{E}}(T_{n,3}^{0,n-4})$.
\end{thm}
\begin{proof}
First, let us calculate the $\mathcal{E}$-spectrum of $T_{n,6}^{a,b,a+1}$ for $2a+b=n-8$ and $a,b\geq 0$. Let $P_7=v_0v_1v_2v_3v_4v_5v_6$ be the unique diametrical path in $T_{n,6}^{a,b,a+1}$. For $i=1,2,\hdots,a$; $j=1,2,\hdots,b$ and $k=1,2,\hdots,a,a+1$, let us assume that $u_i,y_j$ and $w_k$ be the pendant vertices adjacent to $v_2,v_3$ and $v_4$, respectively. Let $U_{a+1}=\{v_1,u_1,u_2,\hdots,u_a\}$, $V_b=\{y_1,y_2,\hdots,y_b\}$ and $V_{a+2}=\{v_4,w_1,w_2,\hdots,w_a,w_{a+1}\}$. It is clear that $\Pi_3:\{v_0\}\cup U_{a+1}\cup \{v_2\}\cup \{v_3\}\cup V_{b}\cup \{v_4\}\cup V_{a+2}\cup \{v_6\}$ is an equitable partition of $\mathcal{E}(T_{n,6}^{a,b,a+1})$ with the quotient matrix
\[Q_{\Pi_3}=
\left [ {\begin{array}{cccccccc}
0 & 0 & 0 & 3 & 4b & 4 & 5(a+2) & 6\\
0 & 0 & 0 & 0 & 0 & 0 & 0 & 5 \\
0 & 0 & 0 & 0 & 0 & 0 & 0 & 4 \\
3 & 0 & 0 & 0 & 0 & 0 & 0 & 3 \\
4 & 0 & 0 & 0 & 0 & 0 & 0 & 4 \\
4 & 0 & 0 & 0 & 0 & 0 & 0 & 0 \\
5 & 0 & 0 & 0 & 0 & 0 & 0 & 0 \\
6 & 5(a+1) & 4 & 3 & 4b & 0 & 0 & 0\\
\end{array} } \right].\]
Now, the characteristic polynomial of $Q_{\Pi_3}$ is $f(x)=x^4(x^4-(50a+32b+161)x^2-(192b+108)x+625a^2+3125a+800ab+1712b+3669)$. Therefore, by Theorem \ref{quo-spec} and Theorem \ref{inertia-even}, it follows that the $\mathcal{E}$-eigenvalues of $T_{n,6}^{a,b,a+1}$ are $0$ with multiplicity $n-4$ and the roots of the polynomial 
\begin{align*}
 p(x)=&x^4-(50a+32b+161)x^2-(192b+108)x+625a^2+3125a+800ab+1712b+3669\\
 =& x^4+(14a-32n+95)x^2+(384a-192n+1428)x+800na+1712n-975a^2-6699a\\
 & -10027 \qquad \text{(since, $b=n-2a-8$)}.
\end{align*}
By Lemma \ref{inertia-even}, it follows that the tree $T_{n,6}^{a,b,a+1}$ has exactly two positive and two negative $\mathcal{E}$-eigenvalues. Hence, $p(x)$ has exactly two positive roots, say $x_1\geq x_2$. Since $b=n-2a-8\geq 0$, therefore $p(5\sqrt{a}-3)=2000(n-2a-7)+3750\sqrt{a}+625>0$, $p(1+\sqrt{32n-39a-156})=-\big(256a+272(n-1)+(128(n-2a-8)+26)\sqrt{32n-39a-156}+195\big) <0$ and $p(3+\sqrt{32n-39a-156})=624(n-2a-6)+234\sqrt{32n-39a-156}+29>0$. Therefore, by the intermediate value theorem, we have $5\sqrt{a}-3<x_2<1+\sqrt{32n-39a-156}$ and $1+\sqrt{32n-39a-156}<x_1<3+\sqrt{32n-39a-156}$. Thus, $E_{\mathcal{E}}(T_{n,6}^{a,b,a+1})>2(\sqrt{32n-39a-156}+5\sqrt{a}-2)$. 

\textbf{Case 1.} If $a<\frac{19n-120}{39}$, then $\sqrt{32n-39a-156}>\sqrt{13n-36}$ and hence by Lemma \ref{bound-spec} and Lemma \ref{bound-second}, we have $E_{\mathcal{E}}(T_{n,6}^{a,b,a+1})>2(\sqrt{32n-39a-156}+5\sqrt{a}-2)>2(\sqrt{13n-36}+\sqrt{2})>2(\xi_1(T_{n,3}^{0,n-4})+\xi_2(T_{n,3}^{0,n-4}))=E_{\mathcal{E}}(T_{n,3}^{0,n-4})$.

\textbf{Case 2.} Let $a>\frac{19n-120}{39}$. Since $b=n-2a-8\geq 0$, therefore 
\begin{align*}
 E_{\mathcal{E}}(T_{n,6}^{a,b,a+1})>&2(\sqrt{32n-39a-156}+5\sqrt{a}-2)\\
 >& 2\bigg(\sqrt{\frac{25n}{2}}+\sqrt{\frac{475n-3000}{39}}-2\bigg).\end{align*}
Hence,
\begin{align*}
&(E_{\mathcal{E}}(T_{n,6}^{a,b,a+1}))^2-(E_{\mathcal{E}}(T_{n,3}^{0,n-4}))^2\\
>& 4\bigg(\sqrt{\frac{25n}{2}}+\sqrt{\frac{475n-3000}{39}}-2\bigg)^2-4\bigg(\sqrt{13n-35+8\sqrt{n-3}}\bigg)^2\\
=&\bigg(\frac{1248(n-3-\sqrt{n-3})+574(n-11)+398}{39}\bigg)+80\sqrt{\frac{n}{2}}\bigg(\sqrt{\frac{19n-120}{78}}-1\bigg)\\
& +80\sqrt{\frac{19n-120}{39}}\bigg(\sqrt{\frac{n}{2}}-1\bigg)+40\sqrt{\frac{19n^2-120n}{78}}\\
>& 0.
\end{align*}
\end{proof}

\begin{thm}\label{diam7-ener}
For $b\geq a \geq 1$ and $a+b=n-8$, $E_{\mathcal{E}}(T_{n,7}^{a,b})>E_{\mathcal{E}}(T_{n,3}^{0,n-4})$.
\end{thm}
\begin{proof}
First, let us calculate the $\mathcal{E}$-spectrum of $T_{n,7}^{a,b}$, where $a+b=n-8$ with $b\geq a \geq 1$. Let $P_8=v_0v_1v_2v_3v_4v_5v_6v_7$ be the unique diametrical path in $T_{n,7}^{a,b}$. Assume that $u_1,u_2,\hdots,u_a$ be the pendant vertices adjacent to $v_2$ and $w_1,w_2,\hdots,w_b$ be the pendant vertices adjacent to $v_3$. If $V_{a+1}=\{v_1,u_1,u_2,\hdots,u_a\}$ and $V_{b+1}=\{v_4,w_1,w_2,\hdots,w_b\}$, then it is easy to see that $\Pi_4:\{v_0\}\cup \{v_1\}\cup V_{a+1}\cup \{v_3\}\cup \{v_4\}\cup V_{b+1}\cup \{v_6\}\cup \{v_7\}$ is an equitable partition of $\mathcal{E}(T_{n,7}^{a,b})$ with the quotient matrix 
\[Q_{\Pi_4}=
\left [ {\begin{array}{cccccccc}
0 & 0 & 0 & 0 & 4 & 5(b+1) & 6 & 7\\
0 & 0 & 0 & 0 & 0 & 0 & 0 & 6 \\
0 & 0 & 0 & 0 & 0 & 0 & 0 & 5 \\
0 & 0 & 0 & 0 & 0 & 0 & 0 & 4 \\
4 & 0 & 0 & 0 & 0 & 0 & 0 & 0 \\
5 & 0 & 0 & 0 & 0 & 0 & 0 & 0 \\
6 & 0 & 0 & 0 & 0 & 0 & 0 & 0 \\
7 & 6 & 5(a+1) & 4 & 0 & 0 & 0 & 0\\
\end{array} } \right].\]
Now, the characteristic polynomial of $Q_{\Pi_4}$ is $f(x)=x^4(x^4-(25a+25b+203)x^2+625ab+1925(a+b)+5929)$. Therefore, by Theorem \ref{quo-spec} and Theorem \ref{Inertia-odd}, the $\mathcal{E}$-spectrum of $T_{n,7}^{a,b}$ is given by 
\[\Spec_{\varepsilon}(T_{n,7}^{a,b})=
	\left\{ {\begin{array}{ccccc}
		\sqrt{\frac{\alpha+\sqrt{\beta}}{2}} & \sqrt{\frac{\alpha-\sqrt{\beta}}{2}} & 0 & -\sqrt{\frac{\alpha-\sqrt{\beta}}{2}} & -\sqrt{\frac{\alpha+\sqrt{\beta}}{2}} \\
		1 & 1 & n-4  & 1 & 1\\
		\end{array} } \right\},\]
where $\alpha=25(a+b)+203$ and $\beta=625(a+b)^2+2450(a+b)-2500ab+17493$. Thus, \begin{align*}
\big(E_{\mathcal{E}}(T_{n,7}^{a,b})\big)^2=& 4\Bigg(\sqrt{\frac{\alpha+\sqrt{\beta}}{2}}+\sqrt{\frac{\alpha-\sqrt{\beta}}{2}}\Bigg)^2 \\
=& 4(\alpha+\sqrt{\alpha^2-\beta})\\
= & 4(25(a+b)+203+\sqrt{7700(a+b)+2500ab+23716})\\
= & 4(25n+3+\sqrt{7700n-37884+2500ab})\qquad (\text{since,~} a+b=n-8).
\end{align*}
Again, by Lemma \ref{energy-bound}, we have $\big(E_{\mathcal{E}}(T_{n,3}^{0,n-4})\big)^2=4(13n-35+8\sqrt{n-3})$. Note that $25n+3+\sqrt{7700n-37884+2500ab}>13n-35+8\sqrt{n-3}$ for $n\geq 5$. Hence, $E_{\mathcal{E}}(T_{n,7}^{a,b})>E_{\mathcal{E}}(T_{n,3}^{0,n-4})$.
\end{proof}

 Now, we obtain a lower bound and an upper bound for the $\mathcal{E}$-spectral radius of $T_{n,3}^{0,n-4}$, which will be used in the subsequent theorems.

\begin{lem}\label{bound-spec}
For $n\geq 5$, $\sqrt{13n-37}<\xi_1(T_{n,3}^{0,n-4})<\sqrt{13n-36}$.
\end{lem}
\begin{proof}
 It follows by Lemma \ref{dstar-spec} that  $$\xi_1(T_{n,3}^{0,n-4})=\sqrt{\frac{13n-35+\sqrt{(13n-37)^2-12(n-4)}}{2}}.$$
 Since $13n-37>12(n-4)$, therefore
\begin{eqnarray*}
\sqrt{(13n-37)^2-12(n-4)} &>& \sqrt{(13n-37)^2-(13n-37)}\\
&=& \sqrt{(13n-37)(13n-38)}\\
&>& 13n-38.
\end{eqnarray*}
Hence, $$\sqrt{\frac{13n-35+\sqrt{(13n-37)^2-12(n-4)}}{2}}>\sqrt{\frac{13n-36+13n-38}{2}}=\sqrt{13n-37}. $$
Again, $\sqrt{(13n-37)^2-12(n-4)} < 13n-37$ implies that $$\sqrt{\frac{13n-35+\sqrt{(13n-37)^2-12(n-4)}}{2}}< \sqrt{\frac{13n-35+13n-37}{2}}=\sqrt{13n-36}. $$
Thus, $\sqrt{13n-37}<\xi_1(T_{n,3}^{0,n-4})<\sqrt{13n-36}$.
\end{proof}

It is known that the tree $T_{n,3}^{0,n-4}$ has the minimum $\mathcal{E}$-spectral radius among all trees on $4\leq n\leq 15$ vertices. In the following theorem, we prove that the tree $T_{n,3}^{0,n-4}$ has the minimum $\mathcal{E}$-spectral radius among all trees on $n\geq 16$ vertices with diameter $d\geq 8$. 

\begin{thm}\label{spec-diam-geq8}
Let $T$ be a tree on $n\geq 16$ vertices with diameter $d\geq 8$. Then, $$\xi_1(T)>\xi_1(T_{n,3}^{0,n-4}).$$
\end{thm}
\begin{proof}
If $T$ is a tree on $n\geq 16$ vertices with diameter $d=8$, then, by Theorem \ref{min-even}, we have $\xi_1(T)\geq \xi_1(T_{n,8}^{\lfloor \frac{n-9}{2}\rfloor,0,\lceil \frac{n-9}{2}\rceil})$. Again, if $T$ is a tree on $n\geq 16$ vertices with diameter $d\geq 9$, then, by Lemma \ref{diam-red-odd}, Lemma \ref{diam-red-even} and Theorem \ref{min-odd}, it follows that $\xi_1(T)\geq \xi_1(T_{n,9}^{\lfloor \frac{n-10}{2}\rfloor,\lceil \frac{n-10}{2}\rceil})$. Therefore, to prove the required result it is sufficient to show that 
\begin{align*}
\xi_1(T_{n,8}^{\lfloor \frac{n-9}{2}\rfloor,0,\lceil \frac{n-9}{2}\rceil})>\xi_1(T_{n,3}^{0,n-4}) \qquad \text{and} \qquad \xi_1(T_{n,9}^{\lfloor \frac{n-10}{2}\rfloor,\lceil \frac{n-10}{2}\rceil})> \xi_1(T_{n,3}^{0,n-4}).    
\end{align*}

\textbf{Case 1.} Let $n\geq 16$ be odd. From Theorem \ref{spec even} and Theorem \ref{spec odd}, it follows that 
\begin{align*}
\xi_1(T_{n,8}^{ \frac{n-9}{2},0, \frac{n-9}{2}})=4+\sqrt{18n-4} \qquad \text{and}
\qquad \xi_1(T_{n,9}^{ \frac{n-11}{2}, \frac{n-9}{2}})=\sqrt{\frac{36n+69+\sqrt{4104n+5913}}{2}} .  
\end{align*}
Note that $18n-4>13n-36$ and $\frac{36n+69+\sqrt{4104n+5913}}{2}>13n-36$. Therefore, by Lemma \ref{bound-spec}, we have $\xi_1(T_{n,8}^{ \frac{n-9}{2},0,\frac{n-9}{2}})=4+\sqrt{18n-4} > \sqrt{13n-36}> \xi_1(T_{n,3}^{0,n-4})$ and $\xi_1(T_{n,9}^{ \frac{n-10}{2}, \frac{n-10}{2}})=\sqrt{\frac{36n+69+\sqrt{4104n+4617}}{2}}> \sqrt{13n-36}> \xi_1(T_{n,3}^{0,n-4})$. 

\textbf{Case 2.} Let $n\geq 16$ be even. From Theorem \ref{spec even}, it follows that $\xi_1(T_{n,8}^{\frac{n-10}{2},0, \frac{n-8}{2}})$ is the largest root of $f(x)=0$, where $$f(x)=x^4-(36n-8)x^2-256x+324n^2-1296n+716.$$
Therefore, 
\begin{align*}
f^{\prime}(x)=4(x^3-(18n-4)x-64), \qquad   f^{\prime \prime}(x)=4(3x^2-18n+4).  
\end{align*}
Thus, $f^{\prime}(x)$ is monotonically decreasing for $x\in (0,\sqrt{\frac{18n-4}{3}})$ and monotonically increasing for $x\in (\sqrt{\frac{18n-4}{3}},\infty)$. Note that $f^{\prime}(0)<0$, therefore $f^{\prime}(x)=0$ has a unique positive root, say $x_0$. Since $(23n+27)^2-(169n^2-974n+1417)=8(45n^2+277n-86)>0$, we have  
\begin{align*}
 f^{\prime}\big(\xi_1(T_{n,3}^{0,n-4})\big)=& 4\Big(\big(\xi_1(T_{n,3}^{0,n-4})\big)^3-(18n-4)\xi_1(T_{n,3}^{0,n-4})-64\Big)\\
 =& 4\xi_1(T_{n,3}^{0,n-4})\Big(\big(\xi_1(T_{n,3}^{0,n-4})\big)^2-(18n-4)\Big)-256\\
 =& 4\xi_1(T_{n,3}^{0,n-4})\bigg(\frac{13n-35+\sqrt{169n^2-974n+1417}}{2}-18n+4\bigg)-256\\
 =& -2\xi_1(T_{n,3}^{0,n-4})\bigg(23n+27-\sqrt{169n^2-974n+1417}\bigg)-256<0.
\end{align*}
This implies that $\xi_1(T_{n,3}^{0,n-4})<x_0<\xi_1(T_{n,8}^{\frac{n-10}{2},0, \frac{n-8}{2}})$.

 Again, by Theorem \ref{spec odd}, we have $\xi_1(T_{n,9}^{ \frac{n-10}{2}, \frac{n-10}{2}})=\sqrt{\frac{36n+69+\sqrt{4104n+4617}}{2}}> \sqrt{13n-36}> \xi_1(T_{n,3}^{0,n-4})$. 
\end{proof}

In the next theorem, we determine the tree with minimum $\mathcal{E}$-energy among all trees on $n$ vertices.

\begin{thm}
Let $T$ be a tree on $n$ vertices. 
\begin{enumerate}
    \item For $n=2,3,4$, the star $K_{1,n-1}$ is the unique tree with minimum $\mathcal{E}$-energy.
    \item For $n\geq 5$, the tree $T_{n,3}^{0,n-4}$ has the minimum $\mathcal{E}$-energy among all trees on $n$ vertices.
\end{enumerate}
\end{thm}
\begin{proof}
For $n=2,3,4$, it is easy to verify that the star $K_{1,n-1}$ is the unique tree with minimum $\mathcal{E}$-energy.

Let $T$ be a tree on $n\geq 5$ vertices other than $T_{n,3}^{0,n-4}$. If $T$ is a tree with diameter $d=2$, then $T\cong K_{1,n-1}$. Now, by Lemma \ref{star-spec} and Lemma \ref{energy-bound}, we have $E_{\mathcal{E}}(K_{1,n-1})=2(n-2+\sqrt{n^2-3n+3})>2\sqrt{13n-35+8\sqrt{n-3}}=E_{\mathcal{E}}(T_{n,3}^{0,n-4})$. 

Let $T$ be a tree on $n\geq 5$ vertices with diameter $d\geq 3$ other than $T_{n,3}^{0,n-4}$. To show  $E_{\mathcal{E}}(T)>E_{\mathcal{E}}(T_{n,3}^{0,n-4})$, by Theorem \ref{Inertia-odd} and Theorem \ref{inertia-even}, it is sufficient to show that $$\xi_1(T)+\xi_2(T)>\xi_1(T_{n,3}^{0,n-4})+\xi_2(T_{n,3}^{0,n-4}).$$

If $T$ is a tree with $ n\leq 15$ vertices, then the proof follows from Theorem \ref{min-gen} and Theorem \ref{second-larg}. Again, if $T$ is a tree with diameter $d=3,4$, then the proof follows from Theorem \ref{diam234-min} and Theorem \ref{second-larg}. So, let us assume that $T$ be a tree on $n\geq 16$ vertices with diameter $d\geq 5$. Now, consider the following cases:

\textbf{Case 1.} Let $d=5$. If $T\cong T_{n,5}^{a,b}$ with $b\geq a\geq 1$ and $a+b=n-6$, then it follows by Theorem \ref{diam5-ener} that $E_{\mathcal{E}}(T)>E_{\mathcal{E}}(T_{n,3}^{0,n-4})$. If $T\ncong T_{n,5}^{a,b}$ with $b\geq a\geq 1$ and $a+b=n-6$, then by Lemma \ref{min-odd}, we have $\xi_1(T)\geq \xi_1(T_{n,5}^{0,n-6})$. Again, by Theorem \ref{spec odd} it follows that
\begin{align*}
\xi_1(T_{n,5}^{0,n-6})=\sqrt{ \frac{16n-21+\sqrt{256n^2-2272n+7541}}{2}}. 
\end{align*}
Since $(39n-313)n+1235>0$ for $n\geq 16$, therefore by Lemma \ref{bound-spec}, we have $\xi_1(T_{n,5}^{0,n-6})=\sqrt{ \frac{16n-21+\sqrt{256n^2-2272n+7541}}{2}}> \sqrt{13n-36}> \xi_1(T_{n,3}^{0,n-4})$. Now, the proof follows from Theorem \ref{second-larg}. 

\textbf{Case 2.} Let $d=6$. If $T\cong T_{n,6}^{a,b,a}$ with $a\geq 1,b\geq 0$ and $2a+b=n-7$, then it follows by Theorem \ref{diam6-ener1} that $E_{\mathcal{E}}(T)>E_{\mathcal{E}}(T_{n,3}^{0,n-4})$. Again, if $T\cong T_{n,6}^{a,b,a+1}$ with $a,b\geq 0$ and $2a+b=n-8$, then by Theorem \ref{diam6-ener2} it follows that $E_{\mathcal{E}}(T)>E_{\mathcal{E}}(T_{n,3}^{0,n-4})$. Let $T$ be a tree other than $T_{n,6}^{a,b,a}$ with $a\geq 1,b\geq 0;2a+b=n-7$ and $T_{n,6}^{a,b,a+1}$ with $a,b\geq 0;2a+b=n-8$. Then it follows by Lemma \ref{even-diam-min} and Theorem \ref{ordering-even} that $\xi_1(T)\geq \xi_1(T_{n,6}^{0,n-7,0})$. Again, from Theorem \ref{spec even} it follows that $\xi_1(T_{n,6}^{0,n-7,0})=3+\sqrt{32n-156}> \sqrt{13n-36}> \xi_1(T_{n,3}^{0,n-4})$. Now, the proof follows from Theorem \ref{second-larg}.

\textbf{Case 3.} Let $d=7$.  If $T\cong T_{n,7}^{a,b}$ with $b\geq a\geq 1$ and $a+b=n-8$, then it follows by Theorem \ref{diam7-ener} that $E_{\mathcal{E}}(T)>E_{\mathcal{E}}(T_{n,3}^{0,n-4})$. If $T\ncong T_{n,7}^{a,b}$ with $b\geq a\geq 1$ and $a+b=n-8$, then by Lemma \ref{min-odd}, we have $\xi_1(T)\geq \xi_1(T_{n,7}^{0,n-8})$. Again, by Theorem \ref{spec odd} it follows that
\begin{align*}
\xi_1(T_{n,7}^{0,n-8})=\sqrt{\frac{25n+3+\sqrt{625n^2-7550n+37884}}{2}}.
\end{align*}
Since $(624n-7400)n+32259>0$ for $n\geq 16$, therefore by Lemma \ref{bound-spec}, we have $\xi_1(T_{n,7}^{0,n-8})=\sqrt{\frac{25n+3+\sqrt{625n^2-7550n+37884}}{2}}> \sqrt{13n-36}> \xi_1(T_{n,3}^{0,n-4})$. Now, the proof follows from Theorem \ref{second-larg}.

\textbf{Case 4.} Let $T$ be a tree on $n\geq 16$ vertices with diameter $d\geq 8$. From Theorem \ref{spec-diam-geq8} and Theorem \ref{second-larg}, it follows that $\xi_1(T)+\xi_2(T)>\xi_1(T_{n,3}^{0,n-4})+\xi_2(T_{n,3}^{0,n-4})$. 
\end{proof}

\bibliographystyle{plain}
\bibliography{ecc-ref}
\end{document}